\documentclass[reqno]{amsart}
\usepackage{amscd}
\usepackage{amssymb}
\usepackage[dvips]{graphics}
\usepackage[margin=1in,includeheadfoot]{geometry}
\usepackage{url}
\usepackage{hyperref}
\usepackage{cite}
\usepackage{color}

\def\beq{\begin{equation}}
\def\eeq{\end{equation}}
\def\ba{\begin{array}}
\def\ea{\end{array}}

\newtheorem{thm}{Theorem}[section]
\newtheorem{lm}[thm]{Lemma}

\theoremstyle{definition}
\newtheorem{rem}[thm]{Remark}

\newtheorem{df}[thm]{Definition}

\theoremstyle{remark}

\begin{document}
\pagestyle{plain}
\title{Boundary $C^{\alpha}$-regularity for solutions of elliptic equations with distributional coefficients}

\author{Jingqi Liang\\
 \small{School of Mathematical Sciences, Shanghai Jiao Tong University}\\
 \small{Shanghai, China}\\\\
\small{ Lihe Wang}\\
 \small{Department of Mathematics, University of Iowa, Iowa City, IA, USA;\\ School of Mathematical Sciences, Shanghai Jiao Tong University}\\
 \small{Shanghai, China}\\\\
\small{Chunqin Zhou}\\
 \small{School of Mathematical Sciences, MOE-LSC, CMA-Shanghai, Shanghai Jiao Tong University}\\
 \small{Shanghai, China}}

\footnote{The third author was partially supported by NSFC Grant 12031012.}

\begin{abstract}
  In this paper, we prove the boundary pointwise $C^{0}$-regularity of weak solutions for Dirichlet problem of elliptic equations in divergence form with distributional coefficients, where the boundary value equals to zero. This is a generalization of the interior case. If $\Omega$ satisfies some measure condition at one boundary point, the bilinear mapping $\langle V\cdot,\cdot\rangle$ generalized by distributional coefficient $V$ can be controlled by a constant sufficiently small, the nonhomogeneous terms satisfy some Dini decay conditions, then the solution is continuous at this point in the $L^{2}$ sense.
\end{abstract}

\maketitle


\section{Introduction}
In this paper, we aim to prove the boundary pointwise continuity for $W^{1,2}$ weak solutions of the following Dirichlet problem with distributional coefficients,
\begin{equation}\label{1}
\left\{
\begin{array}{rcll}
-\Delta u+Vu&=&-\text{div}\vec{f}+g\qquad&\text{in}~~\Omega_{1},\\
u&=&0\qquad&\text{on}~~\partial_{w}\Omega_{1}, \\
\end{array}
\right.
\end{equation}
where $\Omega$ is an open bounded subset of $\mathbb{R}^{n}$ with $n\geq3$.

Maz'ya and Verbitsky in \cite{MV02,MV06} deeply studied the form boundedness of the following Schr\"{o}dinger operator acting from $W^{1,2}(\mathbb{R}^{n})$ to $W^{-1,2}(\mathbb{R}^{n})$:
$$\mathcal{L}=-\text{div}(\nabla)+V,
$$
and general second order differential operator: $\text{div}(A\nabla )+\vec{b}\cdot\nabla +V$, where $A,~\vec{b},~V$ are real or complex valued distributions on $W^{1,2}(\mathbb{R}^{n})$. The definition of form boundedness on operator $\mathcal{L}$ is that:
\begin{equation*}
|\langle \mathcal{L}u,v\rangle|\leq C\|u\|_{W^{1,2}(\mathbb{R}^{n})}\|v\|_{W^{1,2}(\mathbb{R}^{n})}
\end{equation*}
holds for all $u,v\in C_{0}^{\infty}(\mathbb{R}^{n})$, where $C$ does not depend on $u,v\in C_{0}^{\infty}(\mathbb{R}^{n})$. They have already given the sufficient and necessary conditions on the form boundedness of $\mathcal{L}$, i.e. the bilinear form defined by $\langle Vu,v\rangle=\langle V,uv\rangle$ is bounded on $W^{1,2}(\mathbb{R}^{n})\times W^{1,2}(\mathbb{R}^{n})$
if and only if there exists a vector field $\vec{\Gamma}=(\Gamma_{1},\Gamma_{2},\cdots,\Gamma_{n})\in L_{\text{loc}}^{2}(\mathbb{R}^{n})^{n}$ and $\Gamma_{0}\in L_{\text{loc}}^{2}(\mathbb{R}^{n})$ such that
$$V=\text{div}\vec{\Gamma}+\Gamma_{0}
$$
and $|\Gamma_{i}|^{2}(i=0,1,\cdots,n)$ are admissible measure for $W^{1,2}(\mathbb{R}^{n})$, i.e.
$$\int_{\mathbb{R}^{n}}|\Gamma_{i}(x)|^{2}|u(x)|^{2}dx\leq C\|u\|_{W^{1,2}(\mathbb{R}^{n})}^{2},\quad i=0,1,\cdots,n,
$$
where $C$ does not depend on $u\in C_{0}^{\infty}(\mathbb{R}^{n})$. In fact, the form boundedness of $\mathcal{L}$ provides a reasonable way to define $W^{1,2}$ weak solutions of the equations that $\mathcal{L}u$ satisfies.

For  the existence of solutions, Jaye, Maz'ya and Verbitsky \cite{JMV2012} proved that there exists a positive $W^{1,2}_{\text{loc}}$ solution of homogeneous equation: $\mathcal{L}u=0$ under the assumptions that $\langle V,h^{2}\rangle$ has upper and lower bounds with $\lambda<1$ and $\Lambda>0$. The similar existence result was also obtained for operators of $p$-Laplace type in \cite{JMV2013}. For more research on operators with distributional coefficients, such as the operator on fractional Sobolev spaces, infinitesimal form boundedness, we refer readers to \cite{MV021,MV04,MV05} and references therein.

What we are more concerned with is the regularity of weak solutions for equations with distributional coefficients. Our previous work \cite{LWZ2023} discussed interior regularity of solutions for the following nonhomogeneous equation:
\begin{equation}\label{interior}
-\Delta u+Vu=-\text{div}\vec{f}+g\qquad \text{in}~\Omega,
\end{equation}
where $V\in M(W^{1,2}(\Omega)\rightarrow W^{-1,2}(\Omega))$, $\vec{f}\in L^{2}(\Omega)^{n}$ satisfying that $|\vec{f}|^2$ is an admissible measure for $W_{0}^{1,2}(\Omega)$, and $g\in W^{-1,2}(\Omega)$. Now we recall the main result:
If $V,~\vec{f},~g$ satisfy following Dini decay conditions:
\begin{eqnarray*}
& & |\langle V\psi,\varphi\rangle|\leq \omega_{1}(s)\|\psi\|_{L_{r,s}^{1,2}}\|\nabla\varphi\|_{L^{2}(B_{r})}, ~~~ \forall r<s\leq 2r,\\
& & \displaystyle\int_{B_{r}}|\vec{f}|^{2}|\varphi|^{2}dx\leq (\omega_{2}(r))^{2}\displaystyle\int_{B_{r}}|\nabla\varphi|^{2}dx,\\
& & |\langle g,\varphi\rangle|\leq \frac{\omega_{2}(r)}{r^{2}}|B_{r}|^{\frac{n+2}{2n}}\|\nabla\varphi\|_{L^{2}(B_{r})},
\end{eqnarray*}
where $\|\psi\|_{L_{r,s}^{1,2}}=\displaystyle\frac{\|\psi\|_{L^{2}(B_{s})}}{s-r}+\|\nabla \psi\|_{L^{2}(B_{s})}$,  $\omega_{i}(r)$ is Dini modulus of continuity for $i=1,2$, then $u$ is continuous at 0 in the $L^{2}$ sense. We also give a priori estimate to show how the coefficient and nonhomogeneous terms influence the regularity of solutions. This is a framework result. The $\ln$-Lipschitz regularity and H\"{o}lder regularity can also be obtained as corollaries which cover the classical De Giorgi's H\"{o}lder estimates. Furthermore, local regularity can be deduced by this pointwise regularity.

Then it is natural to ask that how about the $C^{0}$-regularity of solutions up to the boundary? The goal of present paper is to answer this question. Here we only consider the Dirichlet problem with zero boundary value. We need to define weak solutions of $(\ref{1})$ and make suitable assumptions on $\Omega,~V,~\vec{f},~g$. It is time to recall the classical result for special case of $(\ref{1})$: $V\in L^{\frac{q}{2}}(\Omega)$ with $q>n$ and $\vec{f}\in L^{q}(\Omega)^{n}$, $g\in L^{\frac{nq}{n+q}}(\Omega)$. Given by Theorem 8.27 and 8.29 in \cite{DT}, it is well-known that if $\Omega$ satisfies an exterior cone condition at a point $x_{0}\in \partial\Omega$, then $u$ is H\"{o}lder continuous at $x_{0}$. Furthermore, if $\Omega$ satisfies a uniform exterior cone condition on a boundary portion $T$, then there exists a constant $1>\alpha>0$ such that $u\in C^{\alpha}(\Omega\cup T)$. Actually boundary H\"{o}lder regularity still holds when the uniform exterior cone condition is relaxed to the following weakest uniform {\bf $(\bar{r},\nu)$-condition}: there exist constants $\nu>0$ and $0<\bar{r}\leq1$ such that for any $0<r\leq \bar{r}$, any $x_{0}\in\partial\Omega$,
\begin{equation}\label{omega}
\frac{|B_{r}(x_{0})\cap\Omega^{c}|}{|B_{r}(x_{0})|}\geq\nu.
\end{equation}
Based on this, we  assume that $(\ref{omega})$ holds for any $x_{0}\in\partial\Omega$ uniformly in this paper. Without loss of generality, we can assume $\bar{r}=1$. Besides, the assumptions on $V,~\vec{f},~g$ are almost parallel to the interior case. Before we state the main theorem, we give some notations and definitions.\\
{\bf Notations:}
(1) $|x|:=\sqrt{\sum\limits_{i=1}^{n} x_{i}^{2}}$: the Euclidean norm of $x=\left(x_{1}, x_{2}, \ldots, x_{n}\right) \in \mathbb{R}^{n}$.

(2) $B_{r}(x_{0}):=\left\{x \in \mathbb{R}^{n}:|x-x_{0}|<r\right\}$.

(3) $B_{r}:=\left\{x \in \mathbb{R}^{n}:|x|<r\right\}$.

(4) $\vec{a}\cdot\vec{b}$: the standard inner product of $\vec{a}, \vec{b} \in \mathbb{R}^{n}$.

(5) $\Omega_{r}=B_{r}\cap\Omega$, $\partial\Omega_{r}$ is the boundary of $\Omega_{r}$, $\partial_{w}\Omega_{r}=\partial\Omega\cap B_{r}$.

(6) $\Omega_{r}(x_{0})=B_{r}(x_{0})\cap\Omega$, $\partial_{w}\Omega_{r}(x_{0})=\partial\Omega\cap B_{r}(x_{0})$.\\

In the sequel, for general domain $\Omega$, we denote by $W^{-1,2}(\Omega)$ the dual space to $W_{0}^{1,2}(\Omega)$, i.e. the class of the bounded linear functional on $ W_{0}^{1,2}(\Omega)$. We write $\langle~,~\rangle$ to denote the pairing between $W^{-1,2}(\Omega)$ and $W_{0}^{1,2}(\Omega)$. Moreover, we say that $g\in W^{-1,2}(\Omega)$, i.e. $g$ is a bounded linear functional on $ W_{0}^{1,2}(\Omega)$, and $\langle g,v\rangle$ satisfies
$$
|\langle g,v\rangle|\leq C\|\nabla v\|_{L^{2}(\Omega)},
$$
for any $v\in W_{0}^{1,2}(\Omega)$, where the constant $C$ depends only on dimension $n$ and $\Omega$ but independent of $v$. We also denote $g(v)=\langle g,v\rangle$.

\begin{df}\label{fbv} For a linear operator $V$ from $W^{1,2}(\Omega)$ to $W^{-1,2}(\Omega)$, if the bilinear mapping$$\langle V\cdot,\cdot\rangle:~W^{1,2}(\Omega)\times W_{0}^{1,2}(\Omega)\rightarrow \mathbb{R}$$
is bounded, i.e. there exists a constant $C$ depends only on dimension $n$ and $\Omega$ such that
\begin{equation}\label{vf1}
|\langle Vu,v\rangle|\leq C\|u\|_{W^{1,2}(\Omega)}\|\nabla v\|_{L^{2}(\Omega)}.
\end{equation} for any $u\in W^{1,2}(\Omega)$ and  $v\in W_{0}^{1,2}(\Omega)$, we call this $V$ as a bounded linear multiplier  from $W^{1,2}(\Omega)$ to $W^{-1,2}(\Omega)$. We also write $M(W^{1,2}(\Omega)\rightarrow W^{-1,2}(\Omega))$ to denote the class of the bounded linear multipliers from $W^{1,2}(\Omega)$ to $W^{-1,2}(\Omega)$.
\end{df}

Next we introduce  a class of admissible measure for $W_{0}^{1,2}(\Omega)$ which is a local version of the admissible measure for $W_0^{1,2}(\mathbb{R}^{n})$ mentioned in the introduction. We refer the readers to see \cite{MV02,MV06} for more details.
\begin{df}\label{dm12}
We say a nonnegative Borel measure $\mu$ on $\Omega$ belongs to the class of admissible measures for $W_{0}^{1,2}(\Omega)$, if $\mu$ obeys the trace inequality,
\begin{equation}\label{m1}
\displaystyle\int_{\Omega}|\varphi|^{2}d\mu\leq C\|\nabla\varphi\|_{L^{2}(\Omega)}^{2},\quad \forall~\varphi\in W_{0}^{1,2}(\Omega),
\end{equation}
where the constant $C$ only depends on $n$ and $\Omega$ but does not depend on $\varphi$. We also write $M(W_0^{1,2}(\Omega)\rightarrow L^{2}(\Omega))$ to denote the class of admissble measures from $W_0^{1,2}(\Omega)$ to $L^{2}(\Omega)$. Especially, for admissible measures $q(x)dx$ with nonnegative density $q\in L^{1}(\Omega)$, we will write $(\ref{m1})$ as
\begin{eqnarray*}
\displaystyle\int_{\Omega}|\varphi|^{2}q(x)dx\leq C\|\nabla\varphi\|_{L^{2}(\Omega)}^{2},\quad \forall~\varphi\in W_{0}^{1,2}(\Omega).
\end{eqnarray*}
\end{df}

Based on above definitions, we can define the weak solutions of $(\ref{1})$.
\begin{df}\label{ws1} Let $\Omega$ be a bounded domain and satisfy $(\bar{r},\nu)$-condition at $0\in\partial\Omega$. We assume $V\in M(W^{1,2}(\Omega_{1})\rightarrow W^{-1,2}(\Omega_{1}))$, $\vec{f}\in L^{2}(\Omega_{1})^{n}$ satisfying that $|\vec{f}|^2$ is an admissible measure for $W_{0}^{1,2}(\Omega_{1})$, and $g\in W^{-1,2}(\Omega_{1})$.
We say $u\in W^{1,2}(\Omega_{1})$ is a weak solution of $(\ref{1})$ in $\Omega_{1}$ if for all $\varphi\in W_{0}^{1,2}(\Omega_{1})$, we have
$$\displaystyle\int_{\Omega_{1}}\nabla u\cdot \nabla\varphi dx+\langle Vu,\varphi\rangle=\displaystyle\int_{\Omega_{1}}\vec{f}\cdot \nabla\varphi dx+\langle g,\varphi\rangle ,
$$
and $u$'s 0-extension is in $W^{1,2}(B_{1})$.
\end{df}

Now we give the main theorem.

\begin{thm}\label{linfty}Let $\Omega$ be a bounded domain and satisfy $(\bar{r},\nu)$-condition at $0\in\partial\Omega$ with $\bar{r}=1$. Assume that $V\in M(W^{1,2}(\Omega_{1})\rightarrow W^{-1,2}(\Omega_{1}))$, $\vec{f}\in L^{2}(\Omega_{1})^{n}$ satisfying that $|\vec{f}|^2$ is an admissible measure for $W_{0}^{1,2}(\Omega_{1})$, and $g\in W^{-1,2}(\Omega_{1})$. Suppose that  $u\in W^{1,2}(\Omega_{1})$ is a weak solution of $(\ref{1})$.
If for any $0<r\leq \displaystyle\frac{1}{2}$ and $\psi\in W^{1,2}(\Omega_{1})$, $\varphi\in W_{0}^{1,2}(\Omega_{1})$ with $\text{supp}\{\varphi\}\subset \overline{\Omega_{r}}$, $\phi\in W_{0}^{1,2}(B_{1})$ with $\text{supp}\{\phi\}\subset \overline{B_{r}}$,
\begin{eqnarray*}
& & |\langle V\psi,\varphi\rangle|\leq C_{V}\|\psi\|_{L_{r,s}^{1,2}(\Omega)}\|\nabla\varphi\|_{L^{2}(\Omega_{r})}, ~~~ \forall r<s\leq 2r,\\
& & |\langle g,\varphi\rangle|\leq \frac{\omega(r)}{r^{2}}|B_{r}|^{\frac{n+2}{2n}}\|\nabla\varphi\|_{L^{2}(\Omega_{r})},\\
& & \displaystyle\int_{\Omega_{r}}|\vec{f}|^{2}|\phi|^{2}dx\leq (\omega(r))^{2}\displaystyle\int_{B_{r}}|\nabla\phi|^{2}dx,
\end{eqnarray*}
where $C_{V}$ is a constant sufficiently small, $\|\psi\|_{L_{r,s}^{1,2}(\Omega)}=\displaystyle\frac{\|\psi\|_{L^{2}(\Omega_{s})}}{s-r}+\|\nabla \psi\|_{L^{2}(\Omega_{s})}$,  $\omega(r)$ is Dini modulus of continuity satisfying $\displaystyle\int_{0}^{1} \frac{\omega(r)}{r} d r<\infty$, then $u$ is continuous at 0 in the $L^{2}$ sense. Moreover, there exists $0<\alpha<1$ such that for any $0<r\leq 2\text{diam}(\Omega)$,
\begin{eqnarray}\label{2}
& & \left(\frac{1}{|B_{r}|}\displaystyle\int_{\Omega_{r}}|u|^{2}dx\right)^{\frac{1}{2}}\leq
CAr^{\alpha}+C
\left(r^{\alpha}\int_{r}^{1}\frac{\omega(s)}{s^{1+\alpha}}ds\right),
\end{eqnarray}
where
\begin{eqnarray*}
A=\left(\displaystyle\frac{1}{|B_{1}|}\int_{\Omega}u^{2}dx
\right)^{\frac{1}{2}}+\frac{4|B_{1}|^{\frac{n+2}{2n}}}{\delta_{0}}\left(\omega(1)
+\frac{1}{1-\lambda}\displaystyle\int_{0}^{1}\frac{\omega(s)}{s}ds\right),
\end{eqnarray*}
Here $C$ is a universal constant $C=C(n)$,  and $\lambda,\delta_{0}$ are the constants in Lemma $\ref{key1}$.
\end{thm}

With respect to the modulus of continuity, we refer the readers to \cite{K1997,K,2} for more details. Here we give some main properties as a complement.
\begin{rem}\label{remark1.3}
Any modulus of continuity $\omega(t)$ is non-decreasing, subadditive, continuous and satisfies $\omega(0)=0$.
\end{rem}

In the following, we give some remarks of Theorem $\ref{linfty}$.
\begin{rem}
From the proof of Theorem $\ref{linfty}$, actually we only need that $C_{V}$ satisfies $C_{V}\leq \displaystyle\frac{\delta_{0}}{128}$ where $\delta_{0}$ is the constant in Lemma $\ref{key1}$.
\end{rem}

\begin{rem}
Theorem $\ref{linfty}$ is a pointwise result, so if $\Omega$ satisfies $(\bar{r},\nu)$-condition at any $x_{0}\in\partial\Omega$ and $V,~\vec{f},~g$ satisfy conditions on $\Omega\cap B_{1}(x_{0})$ respectively, then $u$ is continuous at $x_{0}$ in the $L^{2}$ sense.
\end{rem}

\begin{rem}
If $\bar{r}<1$ even sufficiently small, $V,~\vec{f},~g$ satisfy conditions for $0<r\leq\displaystyle\frac{\bar{r}}{2}$, Theorem $\ref{linfty}$ still holds by scaling. We only need to consider $\tilde{u}(x)=u(\bar{r}x)$, where $x\in B_{1}\cap\tilde{\Omega}$, $\tilde{\Omega}=\{x:\bar{r}x\in\Omega\}$, $\tilde{\omega}(r)=\omega(\bar{r}r)$. It follows that there exists $0<\alpha<1$ such that for any $0<r\leq 2\text{diam}(\tilde{\Omega})$,
\begin{eqnarray*}
& & \left(\frac{1}{|B_{r}|}\displaystyle\int_{\tilde{\Omega}_{r}}|\tilde{u}|^{2}dx\right)^{\frac{1}{2}}\leq
CAr^{\alpha}+C
\left(r^{\alpha}\int_{r}^{1}\frac{\tilde{\omega}(s)}{s^{1+\alpha}}ds\right),
\end{eqnarray*}
where
\begin{eqnarray*}
A=\left(\displaystyle\frac{1}{|B_{1}|}\int_{\tilde{\Omega}}\tilde{u}^{2}dx
\right)^{\frac{1}{2}}+\frac{4|B_{1}|^{\frac{n+2}{2n}}}{\delta_{0}}\left(\tilde{\omega}(1)
+\frac{1}{1-\lambda}\displaystyle\int_{0}^{1}\frac{\tilde{\omega}(s)}{s}ds\right).
\end{eqnarray*}
Scaling back, it follows that for any $0<r\leq 2\text{diam}(\Omega)$,
\begin{eqnarray*}
\left(\frac{1}{|B_{r}|}\displaystyle\int_{\Omega_{r}}|u|^{2}dx\right)^{\frac{1}{2}}\leq
CA\left(\frac{r}{\bar{r}}\right)^{\alpha}+C
\left(r^{\alpha}\int_{r}^{\bar{r}}\frac{\omega(s)}{s^{1+\alpha}}ds\right),
\end{eqnarray*}
where
\begin{eqnarray*}
A=\left(\displaystyle\frac{1}{|B_{\bar{r}}|}\int_{\Omega}u^{2}dx
\right)^{\frac{1}{2}}+\frac{4|B_{1}|^{\frac{n+2}{2n}}}{\delta_{0}}\left(\omega(\bar{r})
+\frac{1}{1-\lambda}\displaystyle\int_{0}^{\bar{r}}\frac{\omega(s)}{s}ds\right).
\end{eqnarray*}
\end{rem}

\begin{rem}\label{ho}
Notice that, in Theorem 1.4, the integral $\displaystyle\int_{r}^{1}\frac{\omega(s)}{s^{1+\alpha}}ds $ may be divergent when $r\to 0^+$. So we only know that $u$ is continuous at $0$ in the $L^2$ sense from the decay of $L^2$ mean value of $u$ on $B_r$, see (\ref{2}).  However, once  the integral $\displaystyle\int_{r}^{1}\frac{\omega(s)}{s^{1+\alpha}}ds $ is convergent when $r\to 0^+$ for some $\omega(r)$, we must get that $u$ is $C^\alpha$  at $0$ in the $L^2$ sense. Especially, if $\omega(r)=Lr^{\alpha_{1}}$ for some $0<\alpha_{1}\leq1$ in Theorem $\ref{linfty}$, then there exists $0<\alpha<\alpha_{1}$ such that $u$ is $C^{\alpha}$ at 0 in the $L^{2}$ sense. Furthermore, for any $0<r\leq 2\text{diam}(\Omega)$,
\begin{eqnarray*}
\left(\frac{1}{|B_{r}|}\displaystyle\int_{\Omega_{r}}|u|^{2}dx\right)^{\frac{1}{2}}\leq
C\left(\|u\|_{L^{2}(\Omega)}+L\right)r^{\alpha},
\end{eqnarray*}
where $C$ depends on $n,\bar{r},\nu,C_{V},\alpha_{1}$.
\end{rem}

If $\Omega$ is a $(\mu)$-type domain and satisfies uniform satisfy $(\bar{r},\nu)$-condition at any boundary point with $\mu+\nu<1$, then combining with Theorem \ref{linfty} and interior H\"{o}lder estimate (Corollary 1.5 in \cite{LWZ2023}), we can give the following global H\"{o}lder estimate. The definition of $(\mu)$-type domain is that there exists $\mu>0$ such that for any $x\in\Omega$, $0<r\leq\text{diam}(\Omega)$, $|\Omega_{r}(x)|\geq\mu|B_{r}(x)|$. In fact, $(\mu)$-type domain guarantees that Campanato embedding theorem can be used.
\begin{thm}\label{global}
Let $\Omega$ be a $(\mu)$-type domain and satisfy uniform satisfy $(\bar{r},\nu)$-condition at any boundary point with $\mu+\nu<1$. Assume that $V\in M(W^{1,2}(\Omega)\rightarrow W^{-1,2}(\Omega))$, $\vec{f}\in L^{2}(\Omega)^{n}$ satisfying that $|\vec{f}|^2$ is an admissible measure for $W_{0}^{1,2}(\Omega)$, and $g\in W^{-1,2}(\Omega)$. Assume that there exist $C_{V}$ sufficiently small, $0<\alpha_{i}<1$, $N_{i}>0$, $i=1,2$ such that for any $x_{0}\in\Omega$, any $0<r\leq \displaystyle\frac{\text{dist}(x_{0},\partial\Omega)}{2}$ and $\psi\in W^{1,2}(\Omega)$, $\varphi\in W_{0}^{1,2}(\Omega)$ with $\text{supp}\{\varphi\}\subset \overline{B_{r}(x_{0})}$,
\begin{eqnarray*}
&&|\langle V\psi,\varphi\rangle|\leq N_{1}s^{\alpha_{1}}\|\psi(\cdot+x_{0})\|_{L_{r,s}^{1,2}}\|\nabla\varphi\|_{L^{2}(B_{r}(x_{0}))},~~~ \forall r<s\leq 2r,
\\
&&
|\langle g,\varphi\rangle|\leq \frac{N_{2}r^{\alpha_{2}}}{r^{2}}|B_{r}|^{\frac{n+2}{2n}}\|\nabla\varphi\|_{L^{2}(B_{r}(x_{0}))},
\\
&&
\displaystyle\int_{B_{r}(x_{0})}|\vec{f}|^{2}|\varphi|^{2}dx\leq N_{2}^{2}r^{2\alpha_{2}}\displaystyle\int_{B_{r}(x_{0})}|\nabla\varphi|^{2}dx,
\end{eqnarray*}
and for any $x_{0}\in\partial\Omega$, $0<r\leq \displaystyle\frac{\bar{r}}{2}$ and $\psi\in W^{1,2}(\Omega_{1}(x_{0}))$, $\varphi\in W_{0}^{1,2}(\Omega_{1}(x_{0}))$ with $\text{supp}\{\varphi\}\subset \overline{\Omega_{r}(x_{0})}$, $\phi\in W_{0}^{1,2}(B_{1}(x_{0}))$ with $\text{supp}\{\phi\}\subset \overline{B_{r}(x_{0})}$,
\begin{eqnarray*}
& & |\langle V\psi,\varphi\rangle|\leq C_{V}\|\psi\|_{L_{r,s,x_{0}}^{1,2}(\Omega)}\|\nabla\varphi\|_{L^{2}(\Omega_{r}(x_{0}))}, ~~~ \forall r<s\leq 2r,\\
& & |\langle g,\varphi\rangle|\leq \frac{N_{2}r^{\alpha_{2}}}{r^{2}}|B_{r}|^{\frac{n+2}{2n}}\|\nabla\varphi\|_{L^{2}(\Omega_{r}(x_{0}))},\\
& & \displaystyle\int_{\Omega_{r}(x_{0})}|\vec{f}|^{2}|\phi|^{2}dx\leq (N_{2}r^{\alpha_{2}})^{2}\displaystyle\int_{B_{r}(x_{0})}|\nabla\phi|^{2}dx,
\end{eqnarray*}
where $\|\psi\|_{L_{r,s,x_{0}}^{1,2}(\Omega)}=\displaystyle\frac{\|\psi\|_{L^{2}(\Omega_{s}(x_{0}))}}{s-r}+\|\nabla \psi\|_{L^{2}(\Omega_{s}(x_{0}))}$, then if $u\in W^{1,2}(\Omega)$ is a weak solution of ($\ref{1}$) in $\Omega$, there exists $\alpha<\min\{\alpha_{1},\alpha_{2}\}$ such that $u\in C^{\alpha}(\overline{\Omega})$. Furthermore,
there exists constant $C=C(n,\alpha_{1},\alpha_{2},N_{1},C_{V},\text{diam}(\Omega),\mu,\nu,\bar{r})$ such that
$$\|u\|_{C^{\alpha}(\overline{\Omega})}\leq C(\|u\|_{L^{2}(\Omega)}+N_{2}).
$$
\end{thm}

\begin{rem}
Actually Theorem $\ref{global}$ covers the classical global H\"{o}lder estimates given by Theorem 8.29 in \cite{DT} where $\vec{f}\in L^{q}(\Omega)^{n}$, $g\in L^{\frac{nq}{n+q}}(\Omega)$, $V\in L^{\frac{q}{2}}(\Omega)$ for some $2n>q>n$. Here we only check the conditions on $V,~\vec{f},~g$ at the boundary point. For the interior case we refer readers to Example 1.8 in \cite{LWZ2023}.

There must be one $\bar{r}$ small enough such that $\bar{r}^{2-\frac{2n}{q}}\|V\|_{L^{\frac{q}{2}}(\Omega)}$ small enough. Then for any $x_{0}\in\partial\Omega$, any $0<r\leq \displaystyle\frac{\bar{r}}{2}$ and $\psi\in W^{1,2}(\Omega_{1}(x_{0}))$, $\varphi\in W_{0}^{1,2}(\Omega_{1}(x_{0}))$ with $\text{supp}\{\varphi\}\subset \overline{\Omega_{r}(x_{0})}$, $\phi\in W_{0}^{1,2}(B_{1}(x_{0}))$ with $\text{supp}\{\phi\}\subset \overline{B_{r}(x_{0})}$, $\langle V\psi,\varphi\rangle$, $\langle g,\varphi\rangle$ can be viewed as $\displaystyle\int_{\Omega_{r}(x_{0})}V\psi\varphi dx$ and $\displaystyle\int_{\Omega_{r}(x_{0})}g\varphi dx$ respectively.
Using H\"{o}lder inequality, Sobolev inequality and Poincar\'{e} inequality it follows that
\begin{eqnarray*}
\left|\langle V\psi,\varphi\rangle\right|&=&\left|\displaystyle\int_{\Omega_{r}(x_{0})}V\psi\varphi dx\right|\\
&\leq&\|V\|_{L^{\frac{n}{2}}(\Omega_{r}(x_{0}))}\|\psi\|_{L^{\frac{2n}{n-2}}(\Omega_{r}(x_{0}))}\|\varphi\|_{L^{\frac{2n}{n-2}}(\Omega_{r}(x_{0}))}\\
&\leq&C(n,q)r^{2-\frac{2n}{q}}\|V\|_{L^{\frac{q}{2}}(\Omega_{r}(x_{0}))}\left(\frac{\|\psi\|_{L^{2}(\Omega_{r}(x_{0}))}}{r}
+\|\nabla\psi\|_{L^{2}(\Omega_{r}(x_{0}))}\right)
\|\nabla\varphi\|_{L^{2}(\Omega_{r}(x_{0}))}\\
&\leq&C(n,q)
\bar{r}^{2-\frac{2n}{q}}\|V\|_{L^{\frac{q}{2}}(\Omega)}
\|\psi\|_{L^{1,2}_{r,s,x_{0}}(\Omega)}\|\nabla\varphi\|_{L^{2}(\Omega_{r}(x_{0}))},
\end{eqnarray*}
for any $r<s\leq 2r$. Similarly,
\begin{eqnarray*}
\left|\langle g,\varphi\rangle\right|=\left|\displaystyle\int_{\Omega_{r}(x_{0})}g\varphi dx\right|
&\leq&C(n,q,\|g\|_{L^{\frac{nq}{n+q}}(\Omega)})r^{-1-\frac{n}{q}}|B_{r}|^{\frac{n+2}{2n}}\|\nabla\varphi\|_{L^{2}(\Omega_{r}(x_{0}))},
\end{eqnarray*}
\begin{eqnarray*}
\displaystyle\int_{\Omega_{r}(x_{0})}|\vec{f}|^{2}\phi^{2}dx
&\leq&C(n,q,\|\vec{f}\|_{L^{q}(\Omega)})r^{2-\frac{2n}{q}}\int_{B_{r}(x_{0})}|\nabla\phi|^{2}dx.
\end{eqnarray*}
We set $C_{V}=C(n,q)
\bar{r}^{2-\frac{2n}{q}}\|V\|_{L^{\frac{q}{2}}(\Omega)}$ sufficiently small, $\omega(r)=C(n,q,\|\vec{f}\|_{L^{q}(\Omega)},\|g\|_{L^{\frac{nq}{n+q}}(\Omega)})r^{1-\frac{n}{q}}$,  then by Theorem $\ref{global}$, it follows that there exists $0<\alpha<1-\displaystyle\frac{n}{q}$ such that $u\in C^{\alpha}(\overline{\Omega})$.
\end{rem}

\begin{rem}
It is worthy to mention that the coefficient $V$ that we consider can not only be functions, but also operators which may not satisfy $\langle Vu,v\rangle=\langle V,uv\rangle$. This is different from the coefficients that Maz'ya and Verbitsky considered in \cite{MV02,MV06}. For example, we set $V=\vec{b}\cdot\nabla$ where $\vec{b}\in L^{n}(\Omega)^{n}$, then $\langle Vu,v\rangle$ can be viewed as $\displaystyle\int_{\Omega_{1}}(\vec{b}\cdot\nabla \psi)\varphi dx$ for any $\psi\in W^{1,2}(\Omega_{1})$, $\varphi\in W_{0}^{1,2}(\Omega_{1})$ with $\text{supp}\{\varphi\}\subset \overline{\Omega_{r}}$. It is obvious that $\langle V\psi,\varphi\rangle\neq\langle V,\psi\varphi\rangle$. 
But on the other hand, by using H\"{o}lder inequality and Sobolev inequality we can calculate that
\begin{eqnarray*}
\left|\langle V\psi,\varphi\rangle\right|&=&\left|\displaystyle\int_{\Omega_{r}}(\vec{b}\cdot\nabla\psi)\varphi dx\right|\\
&\leq&\|\vec{b}\|_{L^{n}(\Omega_{r})}\|\nabla\psi\|_{L^{2}(\Omega_{r})}\|\varphi\|_{L^{\frac{2n}{n-2}}(\Omega_{r})}\\
&\leq&C(n,q)\|\vec{b}\|_{L^{n}(\Omega_{r})}\|\nabla\psi\|_{L^{2}(\Omega_{r})}
\|\nabla\varphi\|_{L^{2}(\Omega_{r})}\\
&\leq&C(n,q)\|\vec{b}\|_{L^{n}(\Omega_{r})}
\|\psi\|_{L^{1,2}_{r,s}(\Omega)}\|\nabla\varphi\|_{L^{2}(\Omega_{r})}.
\end{eqnarray*}
The above inequality implies that $\vec{b}\cdot\nabla\in M(W^{1,2}(\Omega_{1})\rightarrow W^{-1,2}(\Omega_{1}))$ and satisfies the assumption in Theorem $\ref{linfty}$ with $C_{V}=C(n,q)\|\vec{b}\|_{L^{n}(\Omega_{\bar{r}})}$ for $0<r\leq \bar{r}$, $\bar{r}$ sufficiently small.
\end{rem}

The method to prove Theorem $\ref{linfty}$ is similar to the interior case in \cite{LWZ2023}. We still use the compactness method inspired by \cite{BW04,W92} and iteration technique which can be tracked back to \cite{Ca,Ca1989}. We will find a harmonic function $p$ defined in $\Omega_{1}$ with zero boundary value on $\partial_{w}\Omega_{1}$ to approximate $u$, then approximate $u$ by $p(0)$ in the $L^{2}$ sense, the boundary H\"{o}lder estimate will be used in this step. Since $u(0)=p(0)=0$, then we only need to consider the change of $\|u\|_{L^{2}}$ in different scales. It is the main difference from the proof of the interior case. It also illustrates that why we do not need to make Dini decay assumptions on $V$ here: because $u(0)V\equiv0$, then the terms like a constant plus $V$ will not appear on the right hand of the equations during iteration process. Finally the convergence of sum of errors from different scales will lead to the continuity of solution at $0$. We organize the remaining sections as follows. In Section 2, an energy estimate will be given, then we prove an approximation lemma by compactness method and a key lemma that will be used repeatedly in Section 3. The proof of Theorem $\ref{linfty}$ and Theorem $\ref{global}$ will be showed in Section 3.

\section{Energy estimate and compactness lemma}
In this section we will give some preliminary lemmas to prove Theorem $\ref{linfty}$ including energy estimates, and an approximation lemma. We assume in the sequel that $V\in M(W^{1,2}(\Omega_{1}),W^{-1,2}(\Omega_{1}))$, $\vec{f}\in L^{2}(\Omega_{1})^{n}$ satisfying that $|\vec{f}|^2$ is an admissible measure for $W_{0}^{1,2}(\Omega_{1})$, and $g\in W^{-1,2}(\Omega_{1})$. The following lemma will be used in the energy estimate which can be found in Chapter 4 of \cite{HL2011}.
\begin{lm}\label{pee}
Let $h(t)\geq0$ be bounded in $[\tau_{0},\tau_{1}]$ with $\tau_{0}\geq0$. Suppose for any $\tau_{0}\leq t<s \leq\tau_{1}$,
$$h(t)\leq \theta h(s)+\frac{A}{(s-t)^{\alpha}}+B,
$$
for some $\theta\in[0,1)$ and some $A,B\geq0$. Then for any $\tau_{0}\leq t<s \leq\tau_{1}$,
$$h(t)\leq C\left(\frac{A}{(s-t)^{\alpha}}+B\right),
$$
where $C$ is a positive constant depending only on $\alpha$ and $\theta$.
\end{lm}

Next we recall classical boundary H\"{o}lder regularity of solutions, see Theorem 8.27 in \cite{DT}.
\begin{thm}\label{bounho}
Suppose that $\vec{f}\in L^{q}(\Omega)^{n}$, $V,~g\in L^{\frac{q}{2}}(\Omega)$ for some $q>n$. Then if $u$ is a $W^{1,2}(\Omega)$ solution of $(\ref{1})$ in $\Omega$ and $\Omega$ satisfies $\nu$-condition at a point $x_{0}\in\partial\Omega$, then there exists constant $r_{0}>0$ such that for any $0<r\leq r_{0}$,
$$\underset{\Omega\cap B_{r}(x_{0})}{\emph{osc}}u\leq Cr^{\alpha}\left(r_{0}^{-\alpha}\sup_{\Omega\cap B_{r_{0}}(x_{0})}|u|+\|\vec{f}\|_{L^{q}(\Omega)}+\|g\|_{L^{\frac{q}{2}}(\Omega)}\right),
$$
where $C=C(n,q,\|V\|_{L^{\frac{q}{2}}(\Omega)})$, $\alpha=\alpha(n,q,\|V\|_{L^{\frac{q}{2}}(\Omega)})$ are positive constants.
\end{thm}

To begin with, we prove the following energy estimate which will be used in the proof of Lemma $\ref{lm4.2}$.
\begin{lm}[Energy~estimate]\label{ee}Let $\Omega$ be a bounded domain and  $u\in W^{1,2}(\Omega_{1})$ is a weak solution of $(\ref{1})$ in $\Omega_{1}$. Then there exists a sufficiently small positive constant $\varepsilon_{0}$, if
\begin{equation}\label{v1}
|\langle V\psi,\varphi\rangle|\leq \varepsilon_{0}\|\psi\|_{L_{\rho,\theta}^{1,2}(\Omega)}\|\nabla\varphi\|_{L^{2}(\Omega_{\rho})}, ~~~\forall \rho<\theta\leq2\rho,
\end{equation}
\begin{equation}\label{g1}
|\langle g,\varphi\rangle|
\leq \varepsilon_{0}\|\nabla\varphi\|_{L^{2}(\Omega_{\rho})},
\end{equation}
and
\begin{equation}\label{h22}
\displaystyle\int_{\Omega_{\rho}}|\vec{f}|^{2}|\phi|^{2}dx\leq \varepsilon_{0}\displaystyle\int_{B_{\rho}}|\nabla\phi|^{2}dx
\end{equation} for any $0<\rho\leq\displaystyle\frac{1}{2}$, $\psi\in W^{1,2}(\Omega_{1}),~\varphi\in W_{0}^{1,2}(\Omega_{1})$ with $\text{supp}\{\varphi\}\subset \overline{\Omega_{\rho}}$, $\phi\in W_{0}^{1,2}(B_{1})$ with $\text{supp}\{\phi\}\subset \overline{B_{\rho}}$,
we have
$$\displaystyle\int_{\Omega_{t}}|\nabla u|^{2}dx\leq C\left(\displaystyle\int_{\Omega_{1}}|u|^{2}dx\right)\frac{1}{(s-t)^{2}}+1
$$for any $0<t<s\leq \displaystyle\frac{1}{2}$,
where $C$ is a positive constant depending only on $n$ and $\varepsilon_{0}$.
\end{lm}
\begin{proof}
First we take $\eta\in C_{0}^{\infty}(B_{1})$ with $0\leq\eta\leq1$ in $B_{1}$, $\eta=1$ in $B_{t}$, $\eta=0$ in $B_{1}\backslash B_{\frac{t+s}{2}}$, $|\nabla\eta|\leq\displaystyle\frac{c_{0}}{s-t}$.
Note that $\eta^{2}u\in W_{0}^{1,2}(\Omega_{1})$ and its 0-extension is in $W^{1,2}(B_{1})$, then by Definition $\ref{ws1}$, we have
$$\displaystyle\int_{\Omega_{1}}\nabla u\cdot \nabla(\eta^{2}u)dx+\langle Vu,\eta^{2}u\rangle=\displaystyle\int_{\Omega_{1}}\vec{f}\cdot \nabla(\eta^{2}u)dx+\langle g,\eta^{2}u\rangle.
$$
We rewrite the above expression as
$$I_{1}=I_{2}+I_{3}+I_{4}+I_{5},
$$
where
$$I_{1}=\displaystyle\int_{\Omega_{1}}\eta^{2}\nabla u\cdot \nabla udx=\displaystyle\int_{B_{1}}\eta^{2}|\nabla u|^{2}dx,
$$
$$I_{2}=\displaystyle\int_{\Omega_{1}}(2u\eta\vec{f}\cdot \nabla\eta+\eta^{2}\vec{f}\cdot\nabla u)dx,
$$
$$I_{3}=\langle g,\eta^{2}u\rangle,
$$
$$I_{4}=-\displaystyle\int_{\Omega_{1}}2u\eta\nabla u\cdot \nabla \eta dx,
$$
$$I_{5}=-\langle Vu,\eta^{2}u\rangle.
$$
Since $\text{supp}\{\eta\}\subset \overline{B_{\frac{t+s}{2}}}\subset \overline{B_{s}}$, then by using the Cauchy inequality with $\tau<1$ and the assumption $(\ref{v1}),(\ref{g1}),(\ref{h22})$,  we obtain
\begin{eqnarray*}
|I_{2}| &\leq&\displaystyle\int_{\Omega_{1}}(2|\eta\vec{f}||u\nabla\eta|+\eta^{2}|\vec{f}||\nabla u|)dx\\
&\leq&\displaystyle\int_{\Omega_{1}}\left(\eta^{2}|\vec{f}|^{2}+|u|^{2}|\nabla \eta|^{2}+\tau\eta^{2}|\nabla u|^{2}+\frac{1}{4\tau}\eta^{2}|\vec{f}|^{2}\right)dx\\
&\leq&\left(1+\frac{1}{4\tau}\right)\displaystyle\int_{\Omega_{1}}|\vec{f}|^{2}|\eta|^{2}dx+
\displaystyle\int_{\Omega_{1}}|u|^{2}|\nabla \eta|^{2}dx
+\tau\displaystyle\int_{\Omega_{1}}\eta^{2}|\nabla u|^{2}dx\\
&\leq&\left(1+\frac{1}{4\tau}\right)\varepsilon_{0}\displaystyle\int_{B_{1}}|\nabla\eta|^{2}dx+
\tau\displaystyle\int_{\Omega_{s}}|\nabla u|^{2}dx
+\left(\int_{\Omega_{1}}|u|^{2}dx\right)\frac{c_{0}^{2}}{(s-t)^{2}}\\
&\leq&\left(1+\frac{1}{4\tau}\right)\frac{\varepsilon_{0}c_{0}^{2}|B_{1}|}{(s-t)^{2}}+
\tau\displaystyle\int_{\Omega_{s}}|\nabla u|^{2}dx
+\left(\int_{\Omega_{1}}|u|^{2}dx\right)\frac{c_{0}^{2}}{(s-t)^{2}},
\end{eqnarray*}

\begin{eqnarray*}
|I_{3}|&=&|\langle g,\eta^{2}u\rangle|\\
&\leq&\varepsilon_{0}\|\nabla(\eta^{2}u)\|_{L^{2}(\Omega_{s})}\\
&\leq&\varepsilon_{0}\|\eta^{2}\nabla u+2\eta u\nabla\eta\|_{L^{2}(\Omega_{1})}\\
&\leq&\varepsilon_{0}\left(2\displaystyle\int_{\Omega_{1}}\eta^{4}|\nabla u|^{2}dx+8\displaystyle\int_{\Omega_{1}}\eta^{2}|u|^{2}|\nabla \eta|^{2}dx\right)^{\frac{1}{2}}\\
&\leq&\varepsilon_{0}\left(2\displaystyle\int_{\Omega_{1}}\eta^{2}|\nabla u|^{2}dx+8\displaystyle\int_{\Omega_{1}}|u|^{2}|\nabla \eta|^{2}dx\right)^{\frac{1}{2}}\\
&\leq&\frac{\varepsilon_{0}}{2}\left(1+2\displaystyle\int_{\Omega_{1}}\eta^{2}|\nabla u|^{2}dx+8\displaystyle\int_{\Omega_{1}}|u|^{2}|\nabla \eta|^{2}dx\right)\\
&\leq&\varepsilon_{0}\displaystyle\int_{\Omega_{s}}|\nabla u|^{2}dx+\left(4\varepsilon_{0}\int_{\Omega_{1}}|u|^{2}dx\right)\frac{c_{0}^{2}}{(s-t)^{2}}+\frac{\varepsilon_{0}}{2},
\end{eqnarray*}

\begin{eqnarray*}
|I_{4}| &\leq& \displaystyle\int_{\Omega_{1}}2|u\nabla \eta||\eta\nabla u|dx \leq \frac{1}{\tau}\displaystyle\int_{\Omega_{1}}|u|^{2}|\nabla \eta|^{2}dx
+\tau\displaystyle\int_{\Omega_{1}}\eta^{2}|\nabla u|^{2}dx\\
&\leq&\tau\displaystyle\int_{\Omega_{s}}|\nabla u|^{2}dx+\left(\frac{1}{\tau}\displaystyle\int_{\Omega_{1}}|u|^{2}dx\right)\frac{c_{0}^{2}}{(s-t)^{2}},
\end{eqnarray*}
and
\begin{eqnarray*}
|I_{5}|=|\langle Vu,\eta^{2}u\rangle|&\leq& \varepsilon_{0}\|u\|_{L_{\frac{t+s}{2},s}^{1,2}(\Omega)}\|\nabla(\eta^{2} u)\|_{L^{2}(\Omega_{\frac{t+s}{2}})}\\
&\leq&\frac{\varepsilon_{0}}{2}\left(\|u\|_{L_{\frac{t+s}{2},s}^{1,2}(\Omega)}^{2}+\|\nabla(\eta^{2} u)\|_{L^{2}(\Omega_{1})}^{2}\right)\\
&\leq&\frac{\varepsilon_{0}}{2}\left(2\int_{\Omega_{s}}(\frac{4u^{2}}{(s-t)^{2}}+|\nabla u|^{2})dx+2\int_{\Omega_{1}}(\eta^{4}|\nabla u|^{2}+4\eta^{2}u^{2}|\nabla\eta|^{2})dx\right)\\
&\leq&\frac{\varepsilon_{0}}{2}\left(4\int_{\Omega_{s}}|\nabla u|^{2}dx+2\int_{\Omega_{1}}(\frac{4u^{2}}{(s-t)^{2}}+4u^{2}|\nabla\eta|^{2})dx\right)\\
&\leq&2\varepsilon_{0}\int_{\Omega_{s}}|\nabla u|^{2}dx+\left(4\varepsilon_{0}(c_{0}^{2}+1)\displaystyle\int_{\Omega_{1}}|u|^{2}dx\right)\frac{1}{(s-t)^{2}}.
\end{eqnarray*}
Now we combine the estimates $I_{i}(i=1,2,3,4,5)$ to yield that
\begin{eqnarray*}
\displaystyle\int_{\Omega_{t}}|\nabla u|^{2}dx &\leq& \displaystyle\int_{\Omega_{1}}\eta^{2}|\nabla u|^{2}dx = I_{1}\\
&\leq&(3\varepsilon_{0}+2\tau)\int_{\Omega_{s}}|\nabla u|^{2}dx
+\left(((1+\frac{1}{\tau}+8\varepsilon_{0})c_{0}^{2}+4\varepsilon_{0})\displaystyle\int_{B_{1}}|u|^{2}dx\right)\frac{1}{(s-t)^{2}}\\
&&+\left(1+\frac{1}{4\tau}\right)\varepsilon_{0}c_{0}^{2}|B_{1}|
\frac{1}{(s-t)^{2}}+\frac{\varepsilon_{0}}{2}.
\end{eqnarray*}
We choose $\epsilon_0$ and $\tau$ small enough such that $3\varepsilon_{0}+2\tau\leq\displaystyle\frac{1}{2}$, then by Lemma $\ref{pee}$, it follows that $0< t<s \leq\displaystyle\frac{1}{2}$,
\begin{eqnarray*}\displaystyle\int_{\Omega_{t}}|\nabla u|^{2}dx&\leq& C\left(\left(((1+\frac{1}{\tau}+8\varepsilon_{0})c_{0}^{2}+4\varepsilon_{0})\displaystyle\int_{\Omega_{1}}|u|^{2}dx
+\left(1+\frac{1}{4\tau}\right)\varepsilon_{0}c_{0}^{2}|B_{1}|\right)
\frac{1}{(s-t)^{2}}+\frac{\varepsilon_{0}}{2}\right)\\
&\leq&C\left(\displaystyle\int_{\Omega_{1}}|u|^{2}dx+1\right)\frac{1}{(s-t)^{2}}+1,
\end{eqnarray*}
for a positive constant $C$ depending only on $n$ and $\varepsilon_{0}$.
\end{proof}

Next we show the following approximation lemma by the compactness method.
\begin{lm}\label{lm4.2} Assume that $\Omega$ is a bounded domain satisfying $(\bar{r},\nu)$-condition at $0\in\partial\Omega$ with $\bar{r}=1$. For any $\varepsilon>0$, there exists a small $\delta=\delta(\varepsilon)>0$ such that for any weak solution of $(\ref{1})$ in $\Omega_{1}$ with $\displaystyle\frac{1}{|B_{1}|}\displaystyle\int_{\Omega_{1}}u^{2}dx\leq 1$,  and for any $0<\rho\leq\displaystyle\frac{1}{2}$, $\psi\in W^{1,2}(\Omega_{1}),~\varphi\in W_{0}^{1,2}(\Omega_{1})$ with $\text{supp}\{\varphi\}\subset \overline{\Omega_{\rho}}$, $\phi\in W_{0}^{1,2}(B_{1})$ with $\text{supp}\{\phi\}\subset \overline{B_{\rho}}$,
$$|\langle V\psi,\varphi\rangle|\leq \delta \|\psi\|_{L_{\rho,\theta}^{1,2}(\Omega)}\|\nabla\varphi\|_{L^{2}(\Omega_{\rho})},~~~\forall \rho<\theta\leq2\rho,$$
$$|\langle g,\varphi\rangle|\leq \delta\|\nabla\varphi\|_{L^{2}(\Omega_{\rho})},$$
$$\displaystyle\int_{\Omega_{\rho}}|\vec{f}|^{2}\phi^{2}dx\leq\delta^{2}\displaystyle\int_{B_{\rho}}|\nabla \phi|^{2}dx,
$$
there exists a function $p(x)$ defined in $B_{\frac{1}{8}}$, which is a solution of
\begin{eqnarray*}
\left\{
\begin{array}{rcll}
-\Delta p&=&0\qquad&\text{in}~~\Omega_{\frac{1}{8}},\\
p&=&0\qquad&\text{on}~~\partial_{w}\Omega_{\frac{1}{8}}, \\
\end{array}
\right.
\end{eqnarray*}
such that
$$\int_{\Omega_{\frac{1}{8}}}|u-p|^{2}dx\leq\varepsilon^{2}.
$$
\end{lm}
\begin{proof}
We prove it by contradiction. Suppose that there exists $\bar{\varepsilon}>0$, $u_{k},~\vec{f}_{k}$, $g_{k}$ and $V_{k}$ where $\langle V_{k}\cdot,\cdot\rangle$ is bounded on $W^{1,2}(\Omega_{1})\times W_{0}^{1,2}(\Omega_{1})$,
$|\vec{f}_{k}|^{2}$ is an admissible measure on $W^{1,2}(\Omega_{1})$, $g_{k}$ is a bounded linear functional on $W_{0}^{1,2}(\Omega_{1})$, $u_k$ is a solution of
\begin{eqnarray*}
\left\{
\begin{array}{rcll}
-\Delta u_{k}+V_{k}u_{k}&=&-\text{div}\vec{f}_{k}+g_{k}\qquad&\text{in}~~B_{1}\cap\Omega,\\
u_{k}&=&0\qquad&\text{on}~~B_{1}\cap\partial\Omega, \\
\end{array}
\right.
\end{eqnarray*}
satisfying
$$\frac{1}{|B_{1}|}\displaystyle\int_{\Omega_{1}}u_{k}^{2}dx\leq 1,
$$
and for any $0<\rho\leq\displaystyle\frac{1}{2}$, $\psi\in W^{1,2}(\Omega_{1}),~\varphi\in W_{0}^{1,2}(\Omega_{1})$ with $\text{supp}\{\varphi\}\subset \overline{\Omega_{\rho}}$, $\phi\in W_{0}^{1,2}(B_{1})$ with $\text{supp}\{\phi\}\subset \overline{B_{\rho}}$,
\begin{equation}\label{vn1}
|\langle V_{k}\psi,\varphi\rangle|\leq \frac{1}{k}\|\psi\|_{L_{\rho,\theta}^{1,2}(\Omega)}\|\nabla\varphi\|_{L^{2}(\Omega_{\rho})},~~~ \forall \rho<\theta\leq2\rho,
\end{equation}
\begin{equation}\label{g2}
|\langle g_{k},\varphi\rangle|\leq \frac{1}{k}\|\nabla\varphi\|_{L^{2}(\Omega_{\rho})},
\end{equation}
\begin{equation}\label{f1}
\displaystyle\int_{\Omega_{\rho}}|\vec{f}_{k}|^{2}\phi^{2}dx\leq\frac{1}{k^{2}}\displaystyle\int_{B_{\rho}}|\nabla \phi|^{2}dx,
\end{equation}
such that for any function $p$ defined in $B_{\frac{1}{8}}$,  which is  the solution of
\begin{equation}\label{p}
\left\{
\begin{array}{rcll}
-\Delta p&=&0\qquad&\text{in}~~\Omega_{\frac{1}{8}},\\
p&=&0\qquad&\text{on}~~\partial_{w}\Omega_{\frac{1}{8}}, \\
\end{array}
\right.
\end{equation}
the following inequality holds,
\begin{equation}\label{j-1}\int_{\Omega_{\frac{1}{8}}}|u_{k}-p|^{2}dx\geq\bar{\varepsilon}^{2}.
\end{equation}
If letting $k$ be large such that $\displaystyle\frac{1}{k}+\displaystyle\frac{1}{k^{2}}<\varepsilon_{0}$ and taking $t=\displaystyle\frac{1}{4}$, $s=\displaystyle\frac{1}{2}$, then we obtain from  Lemma $\ref{ee}$ that
$$\displaystyle\int_{\Omega_{\frac{1}{4}}}|\nabla u_{k}|^{2}dx\leq C\int_{\Omega_{1}}|u_{k}|^{2}dx+1\leq C.
$$
Denoting $u_{k}$'s 0-extension by $u_{k,0}$, by the definition of weak solution, we have $u_{k,0}\in W^{1,2}(B_{\frac{1}{4}})$. Moreover,
$$\int_{B_{\frac{1}{4}}}(|\nabla u_{k,0}|^{2}+u_{k,0}^{2})dx\leq C.
$$
Hence there exist a subsequence of $\{u_{k,0}\}$, still denoted by $\{u_{k,0}\}$ and $u_{0}\in W^{1,2}(B_{\frac{1}{8}})$ such that
$$u_{k,0}\rightharpoonup u_{0}\quad\text{in}~W^{1,2}(B_{\frac{1}{8}}),
$$
$$u_{k,0}\rightarrow u_{0}\quad\text{in}~L^{2}(B_{\frac{1}{8}}).
$$
Now we denote $u_{0}|_{\Omega_{\frac{1}{8}}}=u$. Then we will show that $u$ itself is a solution of $(\ref{p})$, which is a contradiction. It is obviously that $u_{0}=0$ in $B_{\frac{1}{8}}\cap\Omega^{c}$, so we only need to prove $u$ satisfies the equation. In fact, for any test function $\eta\in W_{0}^{1,2}(\Omega_{\frac{1}{8}})$, we extend $\eta=0$ in $B_{\frac{1}{8}}\backslash \Omega_{\frac{1}{8}}$, still denoted by $\eta$,
then we have
\begin{equation}\label{nc}
\displaystyle\int_{\Omega_{\frac{1}{8}}}\nabla u_{k}\cdot \nabla\eta dx+\langle V_{k}u_{k},\eta\rangle=\displaystyle\int_{\Omega_{\frac{1}{8}}}\vec{f}_{k}\cdot \nabla\eta dx+\langle g_{k},\eta \rangle.
\end{equation}
In $(\ref{f1})$, if we take $\varphi\in C_{0}^{\infty}(B_{1})$ with $\varphi=1$ in $B_{\frac{1}{4}}$, $\varphi=0$ in $B_{1}\backslash B_{\frac{1}{2}}$, $0\leq\varphi\leq1$ in $B_{1}$, we yield that
$$\displaystyle\int_{\Omega_{\frac{1}{4}}}|\vec{f}_{k}|^{2}dx\leq\frac{1}{k^{2}}\displaystyle\int_{B_{1}}|\nabla \varphi|^{2}dx\leq \frac{C}{k^{2}},
$$
where $C$ is a universal constant. By H\"{o}lder inequality, we have
\begin{eqnarray*}
\left|\displaystyle\int_{\Omega_{\frac{1}{8}}}\vec{f}_{k}\cdot \nabla\eta dx\right| \leq
\left(\displaystyle\int_{\Omega_{\frac{1}{8}}}|\vec{f}_{k}|^{2}dx\right)^{\frac{1}{2}}
\left(\displaystyle\int_{\Omega_{\frac{1}{8}}}|\nabla \eta|^{2}dx\right)^{\frac{1}{2}}
\leq \frac{C}{k}\rightarrow0,~~\text{as}~k\rightarrow\infty.
\end{eqnarray*}
Next we apply $(\ref{vn1})$ and $(\ref{g2})$ to obtain
$$|\langle V_{k}u_{k},\eta\rangle|\leq \frac{1}{k}\|u_{k}\|_{L_{\frac{1}{8},\frac{1}{4}}^{1,2}(\Omega)}\|\nabla\eta\|_{L^{2}(\Omega_{\frac{1}{8}})}\leq \frac{C}{k}\rightarrow0,~~\text{as}~k\rightarrow\infty,
$$
$$|\langle g_{k},\eta\rangle|\leq \frac{1}{k}\|\nabla\eta\|_{L^{2}(\Omega_{\frac{1}{8}})}\leq \frac{C}{k}\rightarrow0,~~\text{as}~k\rightarrow\infty.
$$
Now letting  $k\rightarrow\infty$ in $(\ref{nc})$, we have  that
$$\displaystyle\int_{\Omega_{\frac{1}{8}}}\nabla u\cdot \nabla\eta dx=0.
$$
Thus we yield a harmonic function $u$ in $\Omega_{\frac {1}{8}}$, which is contradict to (\ref{j-1}).
\end{proof}

In the following we give a key lemma, which will be used repeatedly in next section.
\begin{lm}[Key Lemma]\label{key1} Let $\Omega$ be a bounded domain satisfying $(\bar{r},\nu)$-condition at $0\in\partial\Omega$ with $\bar{r}=1$. Then there exists $C_{0},~0<\lambda<1,~\delta_{0}>0,~0<\alpha<1$ such that for any weak solution of $(\ref{1})$
with
$$\displaystyle\frac{1}{|B_{1}|}\displaystyle\int_{\Omega_{1}}u^{2}dx\leq 1,
$$
and
$$|\langle V\psi,\varphi\rangle|\leq \delta_{0}\|\psi\|_{L_{\rho,\theta}^{1,2}(\Omega)}\|\nabla\varphi\|_{L^{2}(\Omega_{\rho})},~~~\forall \rho<\theta\leq2\rho,$$
$$|\langle g,\varphi\rangle|\leq \delta_{0}\|\nabla\varphi\|_{L^{2}(\Omega_{\rho})},$$
$$\displaystyle\int_{\Omega_{\rho}}|\vec{f}|^{2}\phi^{2}dx\leq\delta_0^{2}\displaystyle\int_{B_{\rho}}|\nabla \phi|^{2}dx,
$$
for any $0<\rho\leq\displaystyle\frac{1}{2}$, $\psi\in W^{1,2}(\Omega_{1})$ and $\varphi\in W_{0}^{1,2}(\Omega_{1})$ with $\text{supp}\{\varphi\}\subset \overline{\Omega_{\rho}}$, $\phi\in W_{0}^{1,2}(B_{1})$ with $\text{supp}\{\phi\}\subset \overline{B_{\rho}}$, it follows that
$$\left(\frac{1}{|B_{\lambda}|}\int_{\Omega_{\lambda}}|u|^{2}dx\right)^{\frac{1}{2}}\leq\lambda^{\alpha}.
$$
\end{lm}
\begin{proof}
Let $p$ be the function of the previous lemma which satisfies that
$$
\int_{\Omega_{\frac{1}{8}}}|u-p|^{2}dx\leq \varepsilon^2
$$
for some $\varepsilon<1$ to be determined. By the property of harmonic functions, it follows that
$$\|p\|_{L^{\infty}(B_{\frac{1}{16}}\cap\Omega)}\leq C\|p\|_{L^{2}(B_{\frac{1}{8}}\cap\Omega)}\leq C(\varepsilon+|B_{1}|^{\frac{1}{2}})\leq C.
$$
Since $\Omega$ satisfies $\nu$-condition, then by Theorem $\ref{bounho}$, there exist constant $0<r_{0}\leq\displaystyle\frac{1}{16}$, $0<\alpha_{0}<1$ and $C_{0}=C_{0}(n)$ such that for any $x\in B_{r_{0}}\cap\Omega$,
$$|p(x)|=|p(x)-p(0)|\leq C\frac{|x|^{\alpha_{0}}}{r_{0}^{\alpha_0}}\|p\|_{L^{\infty}(B_{\frac{1}{16}}\cap\Omega)}\leq C_{0}\frac{|x|^{\alpha_{0}}}{r_{0}^{\alpha_0}}.
$$

Therefore for each $0<\lambda<r_{0}$, we have
\begin{eqnarray*}
\frac{1}{|B_{\lambda}|}\int_{\Omega_{\lambda}}|u|^{2}dx=\frac{1}{|B_{\lambda}|}\int_{\Omega_{\lambda}}|u-p(0)|^{2}dx
&\leq&\frac{2}{|B_{\lambda}|}\int_{\Omega_{\lambda}}|u-p|^{2}dx
+\frac{2}{|B_{\lambda}|}\int_{\Omega_{\lambda}}|p-p(0)|^{2}dx\\
&\leq&\frac{2\varepsilon^{2}}{|B_{\lambda}|}+2C_{0}^{2}\frac{\lambda^{2\alpha_{0}}}{r_{0}^{2\alpha_{0}}}.
\end{eqnarray*}
Now for any $\alpha<\alpha_{0}$, we take $\lambda$ small enough such that
$$2C_{0}^{2}\frac{\lambda^{2\alpha_{0}}}{r_{0}^{2\alpha_{0}}}\leq\frac{1}{2}\lambda^{2\alpha},
$$
and further we take $\varepsilon$ small sufficiently such that
$$\frac{2\varepsilon^{2}}{|B_{\lambda}|}\leq\frac{1}{2}\lambda^{2\alpha}
$$
and  $\delta_{0}=\delta(\varepsilon)$ in Lemma $\ref{lm4.2}$. Thus  Lemma follows.
\end{proof}

\section{Continuity of solutions under Dini decay conditions}
In this section, we will prove Theorem $\ref{linfty}$. We divide the proof into the following five steps. The first step is to normalize the problem. Secondly we use the key lemma repeatedly to give a iteration result, that is, we approximate the solution in different scales. The next step is to prove the sum of errors from each scale is convergent. Finally scaling back, the continuity of solution follows.\\
\\
{\bf{Proof of Theorem $\ref{linfty}$:}}\\
\
\

$\bf{Step~1: Normalization}$\\

Firstly we give the normalization of the estimates. Without loss of generality, since $C_{V}$ is sufficiently small, we can assume that $C_{V}\leq \displaystyle\frac{\delta_{0}}{128}$ where $\delta_{0}$ is the constant in Lemma $\ref{key1}$.

We set
$$w(x)=\displaystyle\frac{u(x)}{\left(\displaystyle\frac{1}{|B_{1}|}\int_{\Omega}u^{2}dx
\right)^{\frac{1}{2}}+\displaystyle\frac{4|B_{1}|^{\frac{n+2}{2n}}}{\delta_{0}}\left(\omega(1)
+\displaystyle\frac{1}{1-\lambda}\displaystyle\int_{0}^{1}\frac{\omega(s)}{s}ds\right)}
\triangleq\displaystyle\frac{u(x)}{A},\quad x\in B_{1}\cap\Omega,
$$
Then $w(x)$ is a weak solution of
\begin{equation}\label{4-2}
\left\{
\begin{array}{rcll}
-\Delta w+Vw&=&-\text{div}\vec{f}_{A}+g_{A}\qquad&\text{in}~~B_{1}\cap\Omega,\\
w&=&0\qquad&\text{on}~~B_{1}\cap\partial\Omega,\\
\end{array}
\right.
\end{equation}
where $\vec{f}_{A}=\displaystyle\frac{\vec{f}}{A}$, $g_{A}=\displaystyle\frac{g}{A}$.

Then it follows that
\begin{equation}\label{4-3}\left(\frac{1}{|B_{1}|}\displaystyle\int_{\Omega_{1}}w^{2}dx\right)^{\frac{1}{2}}=
\frac{\left(\displaystyle\frac{1}{|B_{1}|}\displaystyle\int_{\Omega_{1}}u(x)^{2}dx\right)^{\frac{1}{2}}}
{A}\leq1.
\end{equation}
Moreover, for any $0<r\leq\displaystyle\frac{1}{2}$, $\varphi\in W_{0}^{1,2}(\Omega_{1})$ with $\text{supp}\{\varphi\}\subset \overline{\Omega_{r}}$, $\phi\in W_{0}^{1,2}(B_{1})$ with $\text{supp}\{\phi\}\subset \overline{B_{r}}$, we have
\begin{eqnarray}\label{4-5}
\left|\langle g_{A},\varphi \rangle\right|
\leq\frac{\omega(r)}{Ar^{2}}|B_{r}|^{\frac{n+2}{2n}}\|\nabla \varphi\|_{L^{2}(\Omega_{r})},
\end{eqnarray}
\begin{eqnarray}\label{4-6}
\int_{\Omega_{r}}|f_{A}|^{2}\phi^{2}dx
\leq\frac{\omega(r)^{2}}{A^{2}}\displaystyle\int_{B_{r}}|\nabla \phi(x)|^{2}dx,
\end{eqnarray}

\
\

$\bf{Step~2:~Iterating~results}$\\

Now, we prove the following claim inductively: there is a nonnegative sequence $\{T_{k}\}_{k=0}^{\infty}$ such that
\begin{equation}\label{ite}
\left(\frac{1}{|B_{\lambda^{k}}|}\int_{\Omega_{\lambda^{k}}}w^{2}dx\right)^{\frac{1}{2}}\leq T_{k},\quad \forall k\geq0,
\end{equation}
where $T_{0}=1$,
\begin{equation}\label{dtk}
T_{k}=\max\left\{\lambda^{\alpha} T_{k-1},
~\frac{16\omega(\lambda^{k})|B_{1}|^{\frac{n+2}{2n}}}{A\delta_{0}}\right\}
\end{equation}
for $k=1,2,\cdots$. 

Firstly by normalized assumption in the {\bf Step 1}, we know that $(\ref{ite})$  holds for $k=0$  since $T_{0}=1$, i.e.
$$\left(\displaystyle\frac{1}{|B_{1}|}\displaystyle\int_{\Omega_{1}}w^{2}dx\right)^{\frac{1}{2}}\leq 1=T_{0}.$$
Now assume by induction that the conclusion is true for $k$. We set
$$\tilde{w}(x)=\frac{w(\lambda^{k}x)}{T_{k}},\quad \tilde{\Omega}=\{x:\lambda^{k}x\in \Omega\}.
$$
It is easy to check that for any $k\geq0$ and any $0<r\leq1$,
$$\frac{|B_{r}\cap\tilde{\Omega}^{c}|}{|B_{r}|}
=\frac{|B_{r\lambda^{k}}\cap\Omega^{c}|}{|B_{r\lambda^{k}}|}\geq\nu.
$$
and $\tilde{w}$ is a weak solution to
\begin{eqnarray*}
\left\{
\begin{array}{rcll}
-\Delta \tilde{w}+\tilde{V}\tilde{w}&=&-\text{div}\vec{\tilde{f}}+\tilde{g}\qquad&\text{in}~~B_{1}\cap\tilde{\Omega},\\
\tilde{w}&=&0\qquad&\text{on}~~B_{1}\cap\partial\tilde{\Omega}, \\
\end{array}
\right.
\end{eqnarray*}
where $\vec{\tilde{f}}(x)
=\displaystyle\frac{\lambda^{k}\vec{f}_{A}(\lambda^{k}x)}{T_{k}}$, $\langle \tilde{V}\cdot,\cdot\rangle$ is bounded on $W^{1,2}(\tilde{\Omega}_{1})\times W_{0}^{1,2}(\tilde{\Omega}_{1})$ satisfying
$$\langle \tilde{V}\psi,\varphi\rangle=\frac{\lambda^{2k}}{\lambda^{nk}}\langle V\tilde{\psi},\tilde{\varphi}\rangle
$$
for $\tilde{\psi}(x)=\psi(\displaystyle\frac{x}{\lambda^{k}})\in W^{1,2}(\Omega_{\lambda^{k}}),~\tilde{\varphi}(x)=
\varphi(\displaystyle\frac{x}{\lambda^{k}})\in W_{0}^{1,2}(\Omega_{\lambda^{k}})$, and $\tilde{g}$ is a bounded linear functional on $W_{0}^{1,2}(\tilde{\Omega}_{1})$ satisfying
$$\langle \tilde{g},\varphi\rangle=\frac{\lambda^{2k}}{T_{k}\lambda^{nk}}\langle g_{A},\tilde{\varphi}\rangle
$$
for $\tilde{\varphi}(x)=
\varphi(\displaystyle\frac{x}{\lambda^{k}})\in W_{0}^{1,2}(\Omega_{\lambda^{k}})$.
Thus by using the inductive assumption, we obtain that
$$\frac{1}{|B_{1}|}\displaystyle\int_{\tilde{\Omega}_{1}}\tilde{w}^{2}dx=
\frac{1}{|B_{1}|}\displaystyle\int_{\tilde{\Omega}_{1}}\frac{|w(\lambda^{k}x)|^{2}}{T_{k}^{2}}dx
=\frac{\displaystyle\frac{1}{|B_{\lambda^{k}}|}\displaystyle\int_{\Omega_{\lambda^{k}}}|w|^{2}dx}{T_{k}^{2}}
\leq1,
$$

\begin{equation}\label{tildev}
\begin{aligned}
\left|\langle\tilde{V}\psi,\varphi \rangle\right|
&=\frac{\lambda^{2k}}{\lambda^{nk}}\left|\langle V\tilde{\psi},\tilde{\varphi}\rangle\right|\\
&\leq\frac{\lambda^{2k}}{\lambda^{nk}}C_{V}\|\tilde{\psi}\|_{L_{\lambda^{k}\rho,\lambda^{k}\theta}^{1,2}(\Omega)}\|\nabla \tilde{\varphi}\|_{L^{2}(\Omega_{\lambda^{k}\rho})}\\
&\leq C_{V}\|\psi\|_{L_{\rho,\theta}^{1,2}(\tilde{\Omega})}\|\nabla \varphi\|_{L^{2}(\tilde{\Omega}_{\rho})}\\
&\leq\frac{\delta_{0}}{128}\|\psi\|_{L_{\rho,\theta}^{1,2}(\tilde{\Omega})}\|\nabla \varphi\|_{L^{2}(\tilde{\Omega}_{\rho})},~~~\forall \rho<\theta\leq2\rho,
\end{aligned}
\end{equation}
for any $0<\rho\leq\displaystyle\frac{1}{2}$, $\psi\in W^{1,2}(\tilde{\Omega}_{1})$, $\varphi\in W_{0}^{1,2}(\tilde{\Omega}_{1})$ with $\text{supp}\{\varphi\}\subset \overline{\tilde{\Omega}_{\rho}}$,
$\tilde{\psi}(x)=\psi(\displaystyle\frac{x}{\lambda^{k}})\in W^{1,2}(\Omega_{\lambda^{k}})$ and  $\tilde{\varphi}(x)=
\varphi(\displaystyle\frac{x}{\lambda^{k}})\in W_{0}^{1,2}(\Omega_{\lambda^{k}})$. Here we have used $C_{V}\leq\displaystyle\frac{\delta_{0}}{128}$. Similarly, since $T_{k}\geq\displaystyle\frac{16\omega(\lambda^{k})|B_{1}|^{\frac{n+2}{2n}}}{A\delta_{0}}$, we have from (\ref{4-5}),
\begin{eqnarray*}
\left|\langle\tilde{g},\varphi \rangle\right|
&=&\frac{\lambda^{2k}}{T_{k}\lambda^{nk}}\left|\langle g_{A},\tilde{\varphi}\rangle\right|\\
&\leq&\frac{\lambda^{2k}}{T_{k}\lambda^{nk}}\frac{\omega(\lambda^{k}\rho)}{A(\lambda^{k}\rho)^{2}}|B_{\lambda^{k}\rho}|^{\frac{n+2}{2n}}\|\nabla \tilde{\varphi}\|_{L^{2}(\Omega_{\lambda^{k}\rho})}\\
&\leq&\frac{\omega(\lambda^{k})}{AT_{k}}|B_{1}|^{\frac{n+2}{2n}}\|\nabla \varphi\|_{L^{2}(\tilde{\Omega}_{\rho})}\\
&\leq&\frac{\delta_{0}}{16}\|\nabla \varphi\|_{L^{2}(\tilde{\Omega}_{\rho})},
\end{eqnarray*}
and from (\ref{4-6}), for any $\phi\in W_{0}^{1,2}(B_{1})$ with $\text{supp}\{\phi\}\subset \overline{B_{\rho}}$, $\tilde{\phi}(x)=
\phi(\displaystyle\frac{x}{\lambda^{k}})\in W_{0}^{1,2}(B_{\lambda^{k}})$,
\begin{eqnarray*}
\int_{\tilde{\Omega}_{\rho}}|\vec{\tilde{f}}|^{2}\phi^{2}dx&=&\frac{\lambda^{2k}}{T_{k}^{2}}\int_{\tilde{\Omega}_{\rho}}|\vec{f}_{A}(\lambda^{k}x)|^{2}
|\varphi(x)|^{2}dx\\
&=&\frac{\lambda^{2k}}{T_{k}^{2}}\frac{1}{\lambda^{nk}}\int_{\Omega_{\lambda^{k}\rho}}|\vec{f}_{A}(t)|^{2}|\tilde{\phi}(t)|^{2}dt\\
&\leq&\frac{\lambda^{2k}}{T_{k}^{2}}\frac{1}{\lambda^{nk}}\frac{\omega(\lambda^{k}\rho)^{2}}{A^{2}}
\int_{B_{\lambda^{k}\rho}}|\nabla\tilde{\phi}(t)|^{2}dt\\
&\leq&\frac{\delta_{0}^{2}}{256}\int_{B_{\rho}}|\nabla\phi(x)|^{2}dx.
\end{eqnarray*}
Thus, we obtain that  for any $0<\rho\leq\displaystyle\frac{1}{2}$, $\psi\in W^{1,2}(\tilde{\Omega}_{1})$, $\varphi\in W_{0}^{1,2}(\tilde{\Omega}_{1})$ with $\text{supp}\{\varphi\}\subset \overline{\tilde{\Omega}_{\rho}}$, for any $\phi\in W_{0}^{1,2}(B_{1})$ with $\text{supp}\{\phi\}\subset \overline{B_{\rho}}$,
$$|\langle \tilde{V}\psi,\varphi\rangle|\leq \delta_{0}\|\psi\|_{L_{\rho,\theta}^{1,2}(\tilde{\Omega})}\|\nabla\varphi\|_{L^{2}(\tilde{\Omega}_{\rho})}, ~~~\forall\rho<\theta\leq2\rho$$
$$|\langle \tilde{g},\varphi\rangle|\leq \delta_{0}\|\nabla\varphi\|_{L^{2}(\tilde{\Omega}_{\rho})},$$
$$\int_{\tilde{\Omega}_{\rho}}|\vec{\tilde{f}}|^{2}\phi^{2}dx\leq\delta_{0}^{2}\int_{B_{\rho}}|\nabla\phi|^{2}dx.
$$
Now we can apply Lemma $\ref{key1}$ for $\tilde{w}$ to obtain that
$$\left(\frac{1}{|B_{\lambda}|}\int_{\tilde{\Omega}_{\lambda}}\tilde{w}^{2}dx\right)^{\frac{1}{2}}\leq\lambda^{\alpha}.
$$
We scale back to get
$$\left(\frac{1}{|B_{\lambda^{k+1}}|}\int_{{\Omega}_{\lambda^{k+1}}}|w(x)|^{2}dx\right)^{\frac{1}{2}}
\leq\lambda^{\alpha} T_{k}\leq T_{k+1}.
$$
Thus we prove the $(k+1)$-th step.

{\bf{Step~3: Prove $\displaystyle\sum\limits_{k=0}^{\infty}T_{k}$ is convergent.}}\\

By the definition of $T_{k}$, $(\ref{dtk})$ implies that
\begin{equation}\label{Tk}
\begin{aligned}
T_{i}
\leq\lambda^{\alpha} T_{i-1}
+\frac{16\omega(\lambda^{i})|B_{1}|^{\frac{n+2}{2n}}}{A\delta_{0}},\quad i=1,2,\cdots.
\end{aligned}
\end{equation}
Hence for any fixed $k\geq1$,
\begin{eqnarray*}
\sum_{i=1}^{k}T_{i}&\leq&\lambda^{\alpha}\sum_{i=1}^{k}T_{i}+\lambda^{\alpha}T_{0}
+\sum_{i=1}^{k}\frac{16\omega(\lambda^{i})|B_{1}|^{\frac{n+2}{2n}}}{A\delta_{0}}\\
&\leq&\lambda^{\alpha}\sum_{i=1}^{k}T_{i}+\lambda^{\alpha}T_{0}
+\sum_{i=0}^{k}\frac{16\omega(\lambda^{i})|B_{1}|^{\frac{n+2}{2n}}}{A\delta_{0}}.
\end{eqnarray*}
Due to for any $k\geq0$,
$$
\sum_{i=0}^{k}\omega(\lambda^{i})\leq
\left(\omega(1)+\frac{1}{1-\lambda}\int_{0}^{1}\frac{\omega(s)}{s}ds\right),
$$
it follows that for any $k\geq1$,
\begin{eqnarray*}
\sum_{i=1}^{k}T_{i}\leq\lambda^{\alpha}\sum_{i=1}^{k}T_{i}+\lambda^{\alpha}T_{0}
+\frac{16|B_{1}|^{\frac{n+2}{2n}}}{A\delta_{0}}\left(\omega(1)+\frac{1}{1-\lambda}\int_{0}^{1}\frac{\omega(s)}{s}ds\right).
\end{eqnarray*}
By the definition of $A$, we obtain
\begin{eqnarray*}
\sum_{i=1}^{k}T_{i}\leq\lambda^{\alpha}\sum_{i=1}^{k}T_{i}+\lambda^{\alpha}T_{0}+4.
\end{eqnarray*}
Consequently we have
$$\sum_{i=1}^{k}T_{i}\leq \frac{4+\lambda^{\alpha}}{1-\lambda^{\alpha}}.
$$
Since $T_{0}=1$, then for any $k\geq0$ it follows
$$\sum_{i=0}^{k}T_{i}\leq\frac{4+\lambda^{\alpha}}{1-\lambda^{\alpha}}+1.
$$
Thus we have shown $\sum\limits_{k=0}^{\infty}T_{k}$ is convergent.\\ 

{\bf{Step~4: Prove $w$ is continuous at 0.}}\\

Using inequality $(\ref{Tk})$ repeatedly, we have for any $k\geq1$,
\begin{eqnarray*}
T_{k}&\leq&\lambda^{k\alpha}T_{0}+\lambda^{k\alpha}\sum_{i=1}^{k}
\frac{16\omega(\lambda^{i})|B_{1}|^{\frac{n+2}{2n}}}{A\delta_{0}\lambda^{i\alpha}}\\
&\leq&\lambda^{k\alpha}T_{0}
+\frac{16\lambda^{k\alpha}|B_{1}|^{\frac{n+2}{2n}}}{A\delta_{0}\lambda^{\alpha}(1-\lambda)}\int_{\lambda^{k}}^{1}\frac{\omega(s)}{s^{1+\alpha}}ds
.
\end{eqnarray*}
Combining with $T_{0}=1$ we have for any $k\geq0$,
\begin{eqnarray*}
T_{k}\leq\lambda^{k\alpha}T_{0}
+\frac{16\lambda^{k\alpha}|B_{1}|^{\frac{n+2}{2n}}}{A\delta_{0}\lambda^{\alpha}(1-\lambda)}\int_{\lambda^{k}}^{1}
\frac{\omega(s)}{s^{1+\alpha}}ds
.
\end{eqnarray*}
It follows that for any $k\geq0$,
\begin{eqnarray*}
\left(\frac{1}{|B_{\lambda^{k}}|}\displaystyle\int_{\Omega_{\lambda^{k}}}|w|^{2}dx\right)^{\frac{1}{2}}
\leq T_{k}
\leq\lambda^{k\alpha}T_{0}
+\frac{16\lambda^{k\alpha}|B_{1}|^{\frac{n+2}{2n}}}{A\delta_{0}\lambda^{\alpha}(1-\lambda)}
\int_{\lambda^{k}}^{1}\frac{\omega(s)}{s^{1+\alpha}}ds.
\end{eqnarray*}
We denote $C_{\lambda,\delta_{0}}=\displaystyle\frac{16|B_{1}|^{\frac{n+2}{2n}}}
{\lambda^{\frac{n}{2}+\alpha}(1-\lambda)\delta_{0}}$ and $\bar{C}_{\lambda,\delta_{0}}=\displaystyle\frac{C_{\lambda,\delta_{0}}}{\lambda^{\alpha}}$. Since for any $0<r\leq1$  there exists a $k\geq0$ such that $\lambda^{k+1}<r\leq \lambda^{k}$, we obtain
\begin{eqnarray*}
\left(\frac{1}{|B_{r}|}\displaystyle\int_{\Omega_{r}}|w|^{2}dx\right)^{\frac{1}{2}}&\leq&
\left(\frac{1}{|B_{\lambda^{k+1}}|}\displaystyle\int_{\Omega_{\lambda^{k}}}|w|^{2}dx\right)^{\frac{1}{2}}\\
&=&\frac{1}{\lambda^{\frac{n}{2}}}\left(\frac{1}{|B_{\lambda^{k}}|}\displaystyle\int_{\Omega_{\lambda^{k}}}|w|^{2}dx\right)^{\frac{1}{2}}\\
&\leq&\frac{\lambda^{k\alpha}}{\lambda^{\frac{n}{2}}}+\frac{C_{\lambda,\delta_{0}}}{A}
\left(\lambda^{k\alpha}\int_{\lambda^{k}}^{1}\frac{\omega(s)}{s^{1+\alpha}}ds\right)\\
&\leq&\frac{r^{\alpha}}{\lambda^{\alpha+\frac{n}{2}}}+\frac{C_{\lambda,\delta_{0}}}{A}
\left(\frac{r^{\alpha}}{\lambda^{\alpha}}\int_{r}^{1}\frac{\omega(s)}{s^{1+\alpha}}ds\right)\\
&\leq&\bar{C}_{\lambda,\delta_{0}}r^{\alpha}+\frac{\bar{C}_{\lambda,\delta_{0}}}{A}
\left(r^{\alpha}\int_{r}^{1}\frac{\omega(s)}{s^{1+\alpha}}ds
\right).
\end{eqnarray*}
By L'Hospital principle, we have
$$\lim\limits_{r\rightarrow0}r^{\alpha}\int_{r}^{1}\frac{\omega(s)}{s^{1+\alpha}}ds=0.
$$
Hence
$$\lim\limits_{r\rightarrow0}\bar{C}_{\lambda,\delta_{0}}r^{\alpha}+\frac{\bar{C}_{\lambda,\delta_{0}}}{A}
\left(r^{\alpha}\int_{r}^{1}\frac{\omega(s)}{s^{1+\alpha}}ds\right)=0,
$$
This implies that $w$ is continuous at 0 in the $L^{2}$ sense.\\

{\bf{Step~5: Scaling back to $u$.}}\\

We notice that $w(x)=\displaystyle\frac{u(x)}{A}$. Then we have for any $0<r\leq 1$,
\begin{eqnarray*}
\left(\frac{1}{|B_{r}|}\displaystyle\int_{\Omega_{r}}|u|^{2}dx\right)^{\frac{1}{2}}
\leq\bar{C}_{\lambda,\delta_{0}}Ar^{\alpha}+\bar{C}_{\lambda,\delta_{0}}
\left(r^{\alpha}\int_{r}^{1}\frac{\omega(s)}{s^{1+\alpha}}ds\right).
\end{eqnarray*}

On the other hand, if $1<r\leq2\text{diam}(\Omega)$, it is easy to calculate that
$$
\left(\frac{1}{|B_{r}|}\displaystyle\int_{\Omega_{r}}|u|^{2}dx\right)^{\frac{1}{2}}
\leq \left(\frac{1}{|B_{1}|}\displaystyle\int_{\Omega}|u|^{2}dx\right)^{\frac{1}{2}}\leq Ar^{\alpha}
$$
Thus we complete the proof of Theorem \ref{linfty}.\\
\\
{\bf{Proof of Theorem $\ref{global}$:}}\\

For any $y\in\Omega$, we set $d=\text{dist}(y,\partial\Omega)$ and we can find $x_{y}\in\partial\Omega$ such that $|y-x_{y}|=d$. Next for any $0<\rho\leq\text{diam}(\Omega)$, we calculate
$$\left(\frac{1}{|\Omega_{\rho}(y)|}\int_{\Omega_{\rho}(y)}|u-u_{\Omega_{\rho}(y)}|^{2}dx\right)^{\frac{1}{2}},
$$
where $u_{\Omega_{\rho}(y)}=\displaystyle\frac{1}{|\Omega_{\rho}(y)|}\int_{\Omega_{\rho}(y)}udx$.

If $\rho\leq d$, we take $R=d$ in Corollary 1.5 of \cite{LWZ2023}, it follows that there exists $\beta=\min\{\alpha_{1},\alpha_{2}\}$ such that
\begin{eqnarray*}
\left(\frac{1}{|\Omega_{\rho}(y)|}\int_{\Omega_{\rho}(y)}|u-u_{\Omega_{\rho}(y)}|^{2}dx\right)^{\frac{1}{2}}&=&
\left(\frac{1}{|B_{\rho}(y)|}\int_{B_{\rho}(y)}|u-u_{B_{\rho}(y)}|^{2}dx\right)^{\frac{1}{2}}\\
&\leq&2\left(\frac{1}{|B_{\rho}(y)|}\int_{B_{\rho}(y)}|u-K_{y}|^{2}dx\right)^{\frac{1}{2}}\\
&\leq&C\left(\left(\displaystyle\frac{1}{|B_{d}(y)|}\int_{B_{d}(y)}u^{2}dx
\right)^{\frac{1}{2}}+N_{2}d^{\alpha_{2}}\right)\left(\frac{\rho}{d}\right)^{\beta}\\
&\leq&C\left(\left(\frac{|B_{2d}(x_{y})|}{|B_{d}(y)|}\displaystyle\frac{1}{|B_{2d}(x_{y})|}\int_{\Omega_{2d}(x_{y})}u^{2}dx
\right)^{\frac{1}{2}}+N_{2}d^{\alpha_{2}}\right)\left(\frac{\rho}{d}\right)^{\beta}.
\end{eqnarray*}
Using Theorem $\ref{linfty}$ at $x_{y}$, there exists $\alpha<\beta$ such that
\begin{eqnarray*}
\left(\displaystyle\frac{1}{|B_{2d}(x_{y})|}\int_{\Omega_{2d}(x_{y})}u^{2}dx
\right)^{\frac{1}{2}}
\leq C(\|u\|_{L^{2}(\Omega)}+N_{2})d^{\alpha}.
\end{eqnarray*}
Combining with above two inequalities, we have
\begin{eqnarray*}
\left(\frac{1}{|\Omega_{\rho}(y)|}\int_{\Omega_{\rho}(y)}|u-u_{\Omega_{\rho}(y)}|^{2}dx\right)^{\frac{1}{2}}&\leq& C
\left((\|u\|_{L^{2}(\Omega)}+N_{2})d^{\alpha}+N_{2}d^{\alpha_{2}}\right)\left(\frac{\rho}{d}\right)^{\beta}\\
&\leq& C
\left((\|u\|_{L^{2}(\Omega)}+N_{2})d^{\alpha}+N_{2}d^{\alpha_{2}}\right)\left(\frac{\rho}{d}\right)^{\alpha}\\
&\leq&C(\|u\|_{L^{2}(\Omega)}+N_{2})\rho^{\alpha}.
\end{eqnarray*}

If $d<\rho\leq\text{diam}(\Omega)$, we can calculate that
\begin{eqnarray*}
\left(\frac{1}{|\Omega_{\rho}(y)|}\int_{\Omega_{\rho}(y)}|u-u_{\Omega_{\rho}(y)}|^{2}dx\right)^{\frac{1}{2}}&\leq&
2\left(\frac{1}{|\Omega_{\rho}(y)|}\int_{\Omega_{\rho}(y)}u^{2}dx\right)^{\frac{1}{2}}\\
&\leq&2\left(\frac{1}{|\Omega_{\rho}(y)|}\int_{\Omega_{\rho+d}(x_{y})}u^{2}dx\right)^{\frac{1}{2}}\\
&=&2\left(\frac{|B_{\rho+d}(x_{y})|}{|\Omega_{\rho}(y)|}\right)^{\frac{1}{2}}
\left(\frac{1}{|B_{\rho+d}(x_{y})|}\int_{\Omega_{\rho+d}(x_{y})}u^{2}dx\right)^{\frac{1}{2}}.
\end{eqnarray*}
Using Theorem $\ref{linfty}$ at $x_{y}$ again, it follows that
\begin{eqnarray*}
\left(\displaystyle\frac{1}{|B_{\rho+d}(x_{y})|}\int_{\Omega_{\rho+d}(x_{y})}u^{2}dx
\right)^{\frac{1}{2}}
\leq C(\|u\|_{L^{2}(\Omega)}+N_{2})(\rho+d)^{\alpha}.
\end{eqnarray*} 
Since $\Omega$ satisfies $(\mu)$-type condition, then $|\Omega_{\rho}(y)|\geq\mu|B_{\rho}(y)|$. Hence combining with above two inequalities, it follows that
\begin{eqnarray*}
\left(\frac{1}{|\Omega_{\rho}(y)|}\int_{\Omega_{\rho}(y)}|u-u_{\Omega_{\rho}(y)}|^{2}dx\right)^{\frac{1}{2}}
&\leq&C\left(\frac{2^{n}}{\mu}\right)^{\frac{1}{2}}
(\|u\|_{L^{2}(\Omega)}+N_{2})(\rho+d)^{\alpha}\\
&\leq&C\left(\frac{2^{n}}{\mu}\right)^{\frac{1}{2}}(\|u\|_{L^{2}(\Omega)}+N_{2})
\left(\frac{\rho+d}{\rho}\right)^{\alpha}\rho^{\alpha}\\
&\leq&C(\|u\|_{L^{2}(\Omega)}+N_{2})\rho^{\alpha}.
\end{eqnarray*}
To summarize, for any $0<\rho\leq\text{diam}(\Omega)$ and $y\in \Omega$, 
$$
\frac{1}{\rho^{\alpha}}\left(\frac{1}{|\Omega_{\rho}(y)|}\int_{\Omega_{\rho}(y)}
|u-u_{\Omega_{\rho}(y)}|^{2}dx\right)^{\frac{1}{2}}\leq C(\|u\|_{L^{2}(\Omega)}+N_{2}).
$$
Then by Campanato embedding theorem, it follows that $u\in C^{\alpha}(\overline{\Omega})$ and
$$\|u\|_{C^{\alpha}(\overline{\Omega})}\leq C(\|u\|_{L^{2}(\Omega)}+N_{2}),
$$
where $C=C(n,\alpha_{1},\alpha_{2},N_{1},C_{V},\text{diam}(\Omega),\mu,\nu,\bar{r})$. Thus we finish the proof of Theorem \ref{global}.


\end{document}